\documentclass[11pt]{article}
\usepackage[english]{babel}
\usepackage{amssymb}
\usepackage{amsthm}
\usepackage{amsmath}
\usepackage{a4wide}
\usepackage{fullpage}
\usepackage{esvect}
\usepackage{hyperref}
\usepackage{verbatim}
\usepackage[dvips]{graphicx}
\usepackage{color}
\usepackage{geometry}
\usepackage{framed}
\usepackage{pgf,tikz}
\usepackage{mathrsfs}
\usetikzlibrary{arrows}
\definecolor{ffqqqq}{rgb}{1.,0.,0.}
\definecolor{qqqqff}{rgb}{0.,0.,1.}
\definecolor{uuuuuu}{rgb}{0.26666666666666666,0.26666666666666666,0.26666666666666666}

\newtheorem{thm}{Theorem}

\newtheorem{lemma}[thm]{Lemma}

\newtheorem{prop}[thm]{Proposition}

\newtheorem{cor}[thm]{Corollary}

\newtheorem{rem}[thm]{Remark}

\newcommand\cH{{\mathcal H}}

\newcommand\cK{{\mathcal K}}

\newcommand\cS{{\mathcal S}}
\newcommand\cT{{\mathcal T}}

\newcommand\od{{\overline d}}

\newcommand{\ignore}[1]{}

\title{A plurality problem with three colors and query size three}

\begin{document}

\author{
D\'aniel Gerbner $^{a}$
\and
D\'aniel Lenger $^b$ 
\and
M\'at\'e Vizer $^{a}$}

\date{
\today}

\maketitle

\begin{center}
$^a$ Alfr\'ed R\'enyi Institute of Mathematics, HAS, Budapest, Hungary\\
\medskip
$^b$ Department of Computer Science, ELTE, Budapest, Hungary \\
\medskip
\small \texttt{gerbner@renyi.hu, leda1648@gmail.com, vizermate@gmail.com}
\medskip
\end{center}

\begin{abstract}

The Plurality problem - introduced by Aigner \cite{A2004} - has many variants. In this article we deal with the following version: suppose we are given $n$ balls, each of them colored by one of three colors. A \textit{plurality ball} is one such that its color class is strictly larger than any other color class. Questioner wants to find a plurality ball as soon as possible or state there is no, by asking triplets (or $k$-sets, in general), while Adversary partition the triplets into color classes as an answer for the queries and wants to postpone the possibility of determining a plurality ball (or stating there is no). 

We denote by $A_p(n,3)$ the largest number of queries needed to ask if both play optimally (and Questioner asks triplets). We provide an almost precise result in case of even $n$ by proving that for $n \ge 4$ even we have
$$\frac{3}{4}n-2 \le A_p(n,3) \le \frac{3}{4}n-\frac{1}{2},$$
and for $n \ge 3$ odd we have
$$\frac{3}{4}n-O(\log n) \le A_p(n,3) \le \frac{3}{4}n-\frac{1}{2}.$$

We also prove some bounds on the number of queries needed to ask for larger $k$.

\end{abstract}


\noindent
{\bf AMS Subj.\ Class.\ (2010)}: 05D99, 91A05  

\section{Introduction}
In combinatorial search theory the so-called \textbf{Plurality problem} is the following problem:
we are given $n$ indexed balls, say $\{1,2,...,n\}(=:[n])$, each colored with one of $c \ge 2$ colors given at the beginning. A \textit{plurality ball} is one such that its color class is strictly larger than any other color class. The aim is either to decide whether there exists a plurality ball or even to show one (if there exists one).  

We can view this problem as a game played between two players: Questioner and Adversary. The role of Questioner is to ask queries and solve the Plurality problem as soon as possible, while Adversary answers the queries and would like to extend the length of the game as long as possible. In this paper we will be interested in the \textbf{maximal length} of the game, when both players play optimally (i.e. the worst case complexity of the plurality problem).  

In the literature there are two main approaches concerning such combinatorial search problems: the first one is the \textit{adaptive}, when queries in the strategy of Questioner can depend on the answers for the previous queries, while in the \textit{non-adaptive} setting Questioner has to ask all the queries in the beginning. In this article we deal with the adaptive setting.

A \textit{model} of the Plurality problem is defined by the number $n$ of balls, the number $c$ of colors, the possible queries, the possible answers, and whether it is adaptive/non-adaptive.

\vspace{3mm}

If $c=2$, this problem coincide with the so called \textit{Majority problem} (where Questioner should find a ball of the color that occurs more than half of the time) that was raised by J. Moore in 1981 in connection with finding the majority vote among $n$ processors using the minimum number of paired comparisons \cite{M1981}. The \textit{pairing} model of the Majority problem is where the query is a pair of balls and the answer is whether their colors coincide or not. It was investigated by Saks and Werman \cite{SW1991}, who gave a precise result for this problem (later a simpler proof was found by Wiener \cite{W2002}), while generalizations for larger query size were investigated in \cite{B2014, DK2015, DKW2012, GKPPVW2017, GV2016}.

\vspace{2mm}

In \cite{A2004} Aigner introduced the Plurality problem as one of the possible generalizations of the Majority problem. Since then many research was carried out about the Plurality problem, see e.g. \cite{A2007, ADM2005, AR2008, AR2009, CGMY2007, DJKKS2007, GKPP2013, ST2008}. In \cite{CGMY2007} it is written, that \textit{"[The Plurality problem] seems to be a more difficult variant"}. It is supported by the fact that the Plurality problem with the pairing model is still unsolved. A lower bound of $3n/2+O(1)$ and an upper bound of $5n/3+O(1)$ were proved in \cite{ADM2005}, while the upper bound is conjectured to be sharp in \cite{A2007}.

\vspace{2mm}

In this article we consider a model of the Plurality problem where the query size is larger. Note that unlike in case of queries of size two, there are several different ways Adversary can answer, which creates several different models. In this article we will be interested in the model where Adversary partitions the query into color classes.

\vspace{2mm}

\textbf{Structure of the paper.} In Section 2 we introduce our model and state our results, in Section 3 we prove our result about query size $k$ (Theorem \ref{general}). In Subsection 4.1 we prove the upper bounds of our main result (Theorem \ref{m1a}), in Subsection 4.2 we prove the lower bound of our main result for $n$ even in full details (lower bound of Theorem \ref{m1a} \textbf{(i)}) and in Subsection 4.3 we sketch the lower bound of our main result for $n$ odd (lower bound of Theorem \ref{m1a} \textbf{(ii)}). We finish our article with some open questions and remarks in Section 5.

\section{Our results}

Let us introduce our model in a more precise form. In this article we are given $n$ indexed balls each colored with one of \textbf{three} given colors. Questioner asks queries of size $k$ in an adaptive manner  and Adversary's answer is a partition of the query set into color classes. Questioner's goal is to find a plurality ball (or state that there is no such ball) as soon as possible, however Adversary wants to postpone the solution as long as possible. We denote the maximum number of queries if both play optimally by $A_p(n,k)$.  

\vspace{2mm}

If Questioner's goal is to determine the color classes (up to permutation), then we denote by $A_c(n,k)$ the analogue quantity with the same constraints. (We note that this problem is the so-called \textbf{Partition problem}.)

\vspace{2mm}

It is easy to see that if Adversary queries all the possible $k$-sets, then he will be able to find the color classes. 

\subsection{General result - Query size of $k$}

We follow the lines of the proofs in \cite{ADM2005} to prove some general bounds on $A_p(n,k)$ and $A_c(n,k)$:

\begin{thm}\label{general} For $n$ even with $n \ge k \ge 2$ we have

$$\frac{n-2}{k-1}+\frac{n}{2(k-1)^2} \le A_p(n,k) \le A_c(n,k) \le \Big{\lceil} \frac{n-1}{k-1} \Big{\rceil}+\Big{\lceil} \frac{n-1}{(k-1)^2}\Big{\rceil}.$$ 

\vspace{2mm}

\noindent
For $n$ odd with $n \ge k \ge 2$ we have

$$ \frac{n-5}{k-1}+\frac{n-k}{2(k-1)^2}\le A_p(n,k) \le A_c(n,k) \le \Big{\lceil} \frac{n-1}{k-1} \Big{\rceil} + \Big{\lceil}\frac{n-1}{(k-1)^2}\Big{\rceil}.$$ 

\end{thm}

\vspace{2mm}

\begin{rem}
We remark that for $k=2$ the lower bound is up to additive constant the same as the one in \cite{ADM2005}. 
\end{rem}

\subsection{Query size of three}

In case of $k=2$ there is a linear gap between the best known lower and upper bound on $A_p(n,2)$ \cite{ADM2005} as we mentioned in the introduction. In Theorem \ref{general} we also have linear gaps between the lower and upper bound on $A_p(n,k)$ for larger $k$.

Surprisingly if the query size is 3, then we can prove an (almost) precise result for the Plurality and the Partition problem in this model in case of $n=4k$ or $n=4k+2$. That means that in the former case we can exactly determine $A_p(n,3)$, while we give two possible values in the latter case. If $n$ is odd we can give the asymptotics of $A_p(n,3)$.

\begin{thm}\label{m1a}

\textbf{(i)} For $n \ge 4$ even we have
$$\frac{3}{4}n-2 \le A_p(n,3)\le A_c(n,3) \le \frac{3}{4}n-\frac{1}{2}.$$
\textbf{(ii)} For $n \ge 3$ odd we have
$$\frac{3}{4}n-\frac12\log_2(n)-5 \le A_p(n,3)\le A_c(n,3) \le \frac{3}{4}n-\frac{1}{2}.$$

\end{thm}

\section{Proof of Theorem \ref{general}}

To prove the upper bound we provide the following algorithm for Questioner (one can see a more detailed version for $k=3$ in the proof of Theorem \ref{m1a}).

Questioner picks a ball $x$ and partitions the remaining elements into $\lceil(n-1)/(k-1)\rceil$ sets of size $k-1$ or smaller. Then he adds arbitrary balls from $[n] \setminus \{x\}$ to the smaller parts to obtain sets of size $k-1$ and asks these sets together with $x$. This way he finds out which balls have the same color as $x$. Among the remaining balls there are only two colors. It is easy to see that if he can make the query family connected (meaning that there is no partition of the remaining balls into two parts such that each (following) query is contained in one of the parts), then he can find out the color classes. There are at most $\lceil(n-1)/(k-1)\rceil$ components, thus $$\Big{\lceil} \frac{\lceil(n-1)/(k-1)\rceil}{k-1} \Big{\rceil} = \Big{\lceil}\frac{n-1}{(k-1)^2}\Big{\rceil}$$ further queries are enough to make the query family connected.

\vspace{3mm}

For the lower bound we generalize the proof of the case $k=2$ by Aigner, De Marco and Montangero (\cite{ADM2005}, Theorem 3.2).

First let us assume that $n$ is even. Adversary maintains the following two conditions during his strategy:

\vspace{2mm}

\textbf{Condition 1:} he always answers in such a way that there is a $3$-coloring of the balls consistent with all the answers (previous answers and the ones just given) such that there are $n/2$ balls of color $1$ and $n/2$ balls of color $2$ and there is no ball of color $3$.  

\vspace{1mm}

\textbf{Condition 2:} he also answers in such a way that not all the $k$ balls of the query are of the same color. He neglects this condition if there are only answers that support Condition 1, but not Condition 2.

\vspace{1mm}

Note that by Condition 1, when the algorithm is finished, the solution for the Plurality problem is that there is no plurality ball.

\vspace{2mm}

Adversary builds a graph $G$ during his algorithm (at a point the graph depends on the queries asked and answers given till that point, however we do not denote this as it is always understandable from the context). The vertices of $G$ are the balls. At the beginning the graph does not contain edges. At the end it has two (consecutively built) sets $A$ and $B$ that both contain $n/2$ vertices, and it is consistent with the answers that $A$ consists of balls of color $1$ and $B$ consists of balls of color $2$. 

At the first query Adversary partitions that query into two non-empty parts $A$ and $B$ and puts a blue spanning tree into both parts, and a red spanning tree between the two parts (balls in $A$ are considered of color $1$ and in $B$ of color $2$). When a query $Q$ is asked at a certain point, Adversary considers the balls of $Q$ that were not asked so far and moves them into $A$ or $B$ such a way that both $A$ and $B$ have sizes at most $n/2$ (including the newly added balls). Let $A'$ be the set of balls of $Q$ in $A$ and $B'$ be the set of balls of $Q$ in $B$. Then Adversary answers that the balls in $A'$ are of one of the colors, the balls in $B'$ are of another color, and there are no balls of the third color in $Q$. He also puts a spanning tree of blue edges into $A'$ and another one into $B'$ (this represents that those vertices surely have the same color). Additionally, in case $Q$ contains vertices from both $A$ and $B$ (after Adversary moves the new vertices there), he also puts a red edge from any ball that was first asked in $Q$ to a ball on the other side (or equivalently a ball of other color). He does it such a way that at most $k-1$ red edges are added. (Note that red edges represent the fact that their endpoints are of different color.) In all cases if the edges he would add are already present, he simply does not add them.

This way either at most $k-1$ blue, or at most $k-2$ blue and at most $k-1$ red edges are added to $G$ at any step. Let the weight of blue edges be $1/(k-1)$ and the weight of red edges $1/(k-1)^2$, this way a query adds at most 1 to the total weight. We will show that the total weight of $G$ is at least $$\frac{n-2}{k-1}+\frac{n}{2(k-1)^2}$$ at the end of the algorithm.

\vspace{2mm}

When the algorithm is finished, there are at least $n-2$ blue edges. Indeed, if $A$ is disconnected, it is possible that one of its components is of color $3$, the other ones are of color $1$, while $B$ consists of color $2$, thus there is a plurality color, but it is also possible that $A$ and $B$ are both monochromatic, thus there is no plurality. 

Also there are at least $n/2$ red edges. Indeed, any ball that is asked first in a query $Q$ is incident to a red edge, except those when the answer was that $Q$ is monochromatic, and its vertices are all in the same part, say $A$. Consider the first such answer where the query $Q$ also contains a new ball. That ball could have been put into $B$ instead, unless $B$ already contained $n/2$ vertices. In that case $B$ reached $n/2$ vertices first, and whenever a new vertex was added to $B$, it was incident to a red edge. Since red edges do not go between vertices of $B$, this means there are at least $n/2$ red edges. The total weight of the at least $n-2$ blue edges and at least $n/2$ red edges is at least $(n-2)/(k-1)+n/2(k-1)^2$, as needed and we are done with the lower bound in case of $n$ even.

\vspace{3mm}

Let us assume now $n$ is odd. Adversary's strategy is similar to the previous one, however in this case we have to solve some technical details that emerge because of the fact that there is no exactly balanced two-coloring of odd many balls. He follows the next two conditions during his strategy:

\vspace{1mm}

\textbf{Condition 1:} he always answers such that there is a consistent $3$-coloring of the balls that contains $(n-1)/2$ balls of color $1$, $(n-1)/2$ balls of color $2$ and one ball of color $3$. 

\vspace{1mm}

\textbf{Condition 2:} he also answers in such a way that not all the $k$ balls of the query are of the same color. He neglects this condition if there are only answers that support Condition 1, but not Condition 2.

\vspace{1mm}

As long as there is a ball that has not appeared in any queries, he answers the same way as in the even case, and blue and red edges are added similarly to the auxiliary graph $G$ (and we count them with the same weights at the end as in the even case).

\vspace{1mm}

When the last ball $x$ is queried by a query $Q$ (if there are more 'last balls' in $Q$ he just picks one of them), Adversary thinks of it as it is the only ball of color $3$ and partition the remaining new balls in such a way that there is $(n-1)/2$ balls of color $1$ (that we call $A$), $(n-1)/2$ balls of color $2$ (that we call $B$) and one ball of color $3$ (it is possible since he did not violate the condition $1$). From that point, as he already decided the colors of the balls, his answers are determined. 

\vspace{1mm}

Now we describe how the definition of the auxiliary graph $G$ changes after $Q$ (including also $Q$). If a query $R$ contains $x$, Adversary puts a spanning tree of blue edges into $A \cap R $ and $B \cap R$ like in the even case, and adds a green edge from a vertex of $A \cap R$ (if such ball exists) to $x$ and another green edge from a vertex of $B \cap R$ (if such ball exists) to $x$. For queries not containing $x$ he does not change the strategy.

\vspace{1mm}

\begin{lemma}\label{oddgreenblue}

If the Questioner can solve the Plurality problem, there are altogether at least $n-5$ blue and green edges. 

\end{lemma}

\begin{proof}

Since $A$, $B$, $\{x\}$ is a coloring consistent with Adversary's answers, Questioner must answer that there are no plurality ball.

If the graph on $A\cup\{x\}$ with the blue and green edges is connected, then there are $\frac{n+1}{2}-1$ edges in it. Assume that it is not connected. Let $A=A_1\cup A_2$, $A_1\cap A_2=\emptyset$, $A_2\neq\emptyset$, such that there are no edges between $A_1\cup\{x\}$ and $A_2$. In this case, $A_1$, $B$, $A_2\cup\{x\}$ is also a coloring consistent with Adversary's answers. If Questioner can solve the Plurality problem, then there should not be any plurality balls.

That means $\max(|A_1|,|A_2\cup\{x\}|)=|B|=\frac{n-1}{2}$. If $|A_1|=|B|$, then $A_2=\emptyset$, which is a contradiction, thus $|A_2\cup\{x\}|=|B|$, so $|A_1|=1$.

Assume that $A_2$ is not connected. Let $A_2=A_3\cup A_4$, $A_3\cap A_4=\emptyset$, $A_3\neq\emptyset$, $A_4\neq\emptyset$, such that there are no edges between $A_3$ and $A_4$. Then $A_1\cup A_3$, $B$, $A_4\cup\{x\}$ is a consistent three-coloring with plurality balls (in $B$). That is a contradiction, so $A_2$ must be connected. Thus, there are at least $\frac{n-3}{2}-1$ edges in $A_2$.

Hence there are at least $\frac{n-5}{2}$ edges in $A\cup\{x\}$. And similarly, there are at least $\frac{n-5}{2}$ edges in $B\cup\{x\}$. We did not count any edges twice, so altogether there are at least $n-5$ blue and green edges.

\end{proof}


We also claim that there are at least $(n-k)/2$ red edges. In fact we prove that there are at least $(n-k)/2$ red edges before querying the last ball $x$. Let us assume first that before choosing $x$, there was a query $R$ where the answer was that it is monochromatic. In that case there are still at least $n/2$ red edges in $G$ at that point by the same argument that we used in the even case. Let us consider now the case that there was no monochromatic answer before choosing $x$. Then at that point at least $n-k$ balls that appeared in the earlier queries and all of them have an incident red edge also by the same argument we used in the even case.

\vspace{2mm}

The weight of the red and blue are the same as in the even case and let the weight of a green edge be $1/(k-1)$. The weight of any query will be at most 1. By the previous argument the total weight in $G$ at the end of the algorithm is at least 

$$\frac{n-5}{k-1}+\frac{n-k}{2(k-1)^2},$$

\noindent
that proves the statement in the case of odd $n$.

\section{Proof of Theorem \ref{m1a}}

\subsection{Proof of the upper bound of Theorem \ref{m1a} \textbf{(i),   (ii)}}


We provide a more precise version of the proof that we provided in the proof of the upper bound of Theorem \ref{general}. To prove the upper bound we divide Questioner's algorithm into two phases:

\vspace{3mm}

\textbf{Phase 1}: Questioner chooses an arbitrary ball $x \in [n]$, and partitions all the other balls into pairs. (If $n$ is even, there is one extra ball $x' \in [n]$, that is not an element of a pair.) He asks all the queries that contain a pair and $x$. If $n$ is even, he also asks the query $\{x, x', y \}$, where $y \in [n] \setminus \{x,x'\}$ is an arbitrary ball. So far he asked $\left\lceil\frac{n-1}{2}\right\rceil$ queries.

After these queries he makes an auxiliary graph $G$, whose vertex set is $[n]$ minus those balls that turned out to have the same color as $x$ (including $x$) during Phase 1. If a query was $\{x, z, z'\}$, and the answer was that $x$ has different color than $z$ and $z'$, then he adds the edge $(z,z')$ to $G$. He colors it red if $z$ and $z'$ have different colors, and blue if $z$ and $z'$ have the same color. Let $C_i \subset [n]$ be the set of balls with color $i$ (for $i=1, 2, 3$). We can assume that $x$ has color $1$. So we have that $V(G)=C_2\cup C_3=[n]\setminus C_1$.

Note that $G$ has at most $$n-1-\left\lceil\frac{n-1}{2}\right\rceil=\left\lfloor\frac{n-1}{2}\right\rfloor$$ components (as by every query, the number of components of $G$ decreased by at least one).

\vspace{2mm}

Questioner's next goal is to make $G$ connected with some more queries. If it is so, then he is able to separate the remaining two color classes.

\vspace{4mm}

\textbf{Phase 2:} first he asks a query that contains balls from three different components of $G$. If he asks $\{x_1, x_2, x_3\}$, then he adds the edges $(x_1,x_2), \ (x_1,x_3), \ (x_2,x_3)$ to $G$. He colors the edge $(x_i, x_j)$ blue if $x_i$ and $x_j$ have the same color, and red if they have different colors. After the answer, the three components become one. Then in the new graph we continue this procedure till we can.

If at one step there are only two components, then it must be the last query and he asks one ball from the first one, and two from the second one.

(If there are at most $2$ balls in $G$, then he asks at most $1$ queries that contains all the balls in $G$ and some balls with color $1$.)

\vspace{2mm}

Note that we had at most $\left\lfloor\frac{n-1}{2}\right\rfloor$ components of the beginning of Phase 2, and their number decreases by two for all but the last of the queries, thus he asked at most $$\left\lceil\frac{\left\lfloor\frac{n-1}{2}\right\rfloor-1}{2}\right\rceil$$ queries during Phase 2.

\vspace{2mm}

Now the auxiliary graph is connected. Let $y$ be a vertex of it, and say it has color $2$. For every vertex $z$, there is a path from $y$ to $z$. Then the parity of the number of the red edges in this (and every) path between $y$ and $z$ shows us the color of $z$. So we are able to determine the set $C_2$, $C_3$.

Altogether he asked at most $$\left\lceil\frac{n-1}{2}\right\rceil+\left\lceil\frac{\left\lfloor\frac{n-1}{2}\right\rfloor-1}{2}\right\rceil$$ queries. One can easily calculate that if $n\equiv 0, 1, 2, 3\pmod 4$, then it is equal to $$\frac{3}{4}n-1, \ \frac{3}{4}(n-1), \ \frac{3}{4}(n-2)+1, \ \frac{3}{4}(n-3)+1,$$ respectively. As $$\frac{3}{4}(n-2)+1=\frac{3}{4}n-\frac{1}{2}$$ is the largest of these numbers, we are done with the upper bound.

\qed

\subsection{Proof of the lower bound in Theorem \ref{m1a} \textbf{(i)}}


Now, we prove the lower bound. We give a strategy for Adversary which forces any algorithm that can solve the Plurality problem to ask at least $$\frac{3}{4}n-2$$ questions for every even $n$. 

Let $\mathcal{S}$ denote the set of the balanced two-colorings of the $n$ \textbf{balls}. The most important point in Adversary's strategy is that he takes care that there is always a coloring in $\mathcal{S}$ that is consistent with the previous answers. He will never violate this.




He makes an auxiliary graph $G$ to code the information that one can gain from his answers in the following way: 

\vspace{2mm}

$\bullet$ The vertices $G$ are the balls. 

\vspace{1mm}

$\bullet$ At the beginning, there are no edges in the graph. If he gets the answer for a query $\{i,j,k\}$ that balls $i$ and $j$ have the same color, and $k$ has a different color, than he draws the edges $ik$ and $jk$ with red, and $ij$ with blue. If he gets the answer that balls $i$, $j$ and $k$ all have the same color, than he draws two of the three edges $ij$, $ik$ and $jk$ with blue.

\vspace{1mm}

$\bullet$ At each point $G$ contains all the edges that comes from the answers received so far. (We note that it would be more precise to denote the current state of $G$, that is what queries were asked so far, but we omit this since it won't be misleading at any point of the proof.) 

\vspace{2mm}

Later, we will think of these red and blue colors as weights. Red edges will have weight $\frac14$, blue edges will have weight $\frac12$, so  the total weight of every answer is 1.

\vspace{2mm}

We denote by $G_R$ the graph spanned by the red edges of $G$, and by $G_B$ the graph spanned by the blue edges of $G$. By the strategy, and the definition of the edges, $G_R$ is a bipartite graph, and $G_B$ has at least two components. Note also that each ball in a component of $G_B$ has the same color.


\begin{prop}\label{endblue}
If Questioner can solve the Plurality problem at certain point, then $G$ has at least $n-2$ blue edges at that point.
\end{prop}

\begin{proof}
By the strategy of Adversary there is a coloring from $\mathcal{S}$ that is consistent with the answers, so the only solution for the Plurality problem can be that there is no plurality ball.

Again, we know that there is at least one coloring (of the balls) from $\mathcal{S}$ that is consistent with $G$. Let $X$ and $Y$ the two color classes of that coloring. Assume that $G_B$ has more than two components. Note that each component of $G_B$ is a subset of either $X$ or $Y$. $G_B$ has more than two components, so without loss of generality we can assume that there are no blue edges between $X_1$ and $X_2$ (where $X=X_1\cup X_2$, $X_1\cap X_2=\emptyset$, $X_1\neq\emptyset$, $X_2\neq\emptyset$). In this case $X_1$, $X_2$ and $Y$ can be the three color class, and any ball from $Y$ is a plurality ball.

This means that if Questioner can solve the Plurality problem against Adversary's strategy, then $G_B$ has only two components at the end. Thus there are at least $n-2$ blue edges.

\end{proof}

\vspace{2mm}

Our next goal is to show that there are at least $n-4$ red edges.

\vspace{1mm}

If we can prove that, we can consider the weights of these red and blue colors. As we mentioned, red edges have weight $\frac14$, blue edges have weight $\frac12$, so every answer has a total weight of 1. Thus, the total number of queries is at least the sum of the weights of the edges, so at least
$$\frac{1}{2}(n-2)+\frac{1}{4}(n-4)=\frac{3}{4}n-2,$$
and that will prove Theorem \ref{m1a}.


\vspace{2mm}
Now we start to work towards this goal. First let us introduce the following notation.
Consider a component $D$ of $G_R$. It is a bipartite graph, call the two classes of vertices $D_X$ and $D_Y$. Let $$d(D):=\big{|} |D_X|-|D_Y| \big{|}.$$ Then we say $d(D)$ is the \textit{imbalance} of component $D$.

\vspace{4mm}

\textbf{Strategy of Adversary:}

\vspace{2mm}

Now we will specify how Adversary should answer for a certain query. First we give 3 conditions (we mentioned the first one earlier) that will be the main lines of the strategy of Adversary. Then we specify the strategy more. 

\vspace{2mm}

\textbf{Condition 1}: as we mentioned, there must be at least one balanced two-colorings consistent with the answers. Adversary will never violate this condition. 

\textbf{Condition 2}: Adversary answers that two of the three balls have the same color, and the third one has a different. 
(He answers according to Condition 2, if he does not violate Condition 1 with the answer.)

\textbf{Condition 3}: Adversary makes components with small imbalance, as described below. 
(He does this if he can so that he does not violate Condition 1 with his answer.)

\vspace{3mm}

We will consider the first time, when Adversary has to violate Condition 2 or 3. It is not necessarily the end of the algorithm. After that violation happened, Adversary can violate again Condition 2 and 3 with any answer, but not Condition 1. (So one can think that the algorithm consists of two phases. During Phase 1, Adversary does not violate Condition $1,2$ and $3$ and during Phase 2, he does not violate Condition $1$.)


\vspace{2mm}

Now we give Adversary's strategy in more details (in the description of the strategy $G$ means the auxiliary graph before a specific query):

\vspace{2mm}

$\bullet_1$ If the query contains three balls from the same component of $G$, then the answer is determined by the two-coloring of that component of $G$.

\vspace{2mm}

$\bullet_2$ If the query $\{x_1,x_2,x_3\}$ contains three balls from exactly two components of $G$: $A$ and $B$, then two of the three balls are in the same component and we can assume that  $x_1, x_2\in A$ and $x_3\in B$. There are two cases:

\vspace{2mm}

\hspace{4mm}$\bullet_{2.1}$ First assume that $x_1$ and $x_2$ are in the same color class of $A$. In this case, Adversary answers that $x_1$ and $x_2$ has the same color, and $x_3$ has a different color (if that answer does not violate Condition 1. If it does, he has to answer that the three balls have the same color).

\vspace{2mm}

\hspace{4mm}$\bullet_{2.2}$ Now assume that $x_1$ and $x_2$ are in different color classes of $A$. Thus the answer will be that $x_1$ and $x_2$ have different colors, and $x_3$ has a color such that this answer does not violate Condition 1. Note that before this query by our assumption, there is at least one such balanced two-coloring. This will be the answer (so Adversary in this case can answer the query without violating Condition 1 or 2).

Note that if $d(A)=a$, and $d(B)=b$, then the imbalance of the new component in any cases of $\bullet_2$ will be either $a+b$ or $a-b$ or $b-a$.


\vspace{3mm}

$\bullet_3$ If the query contains three balls from three different components: $A$, $B$, $C$, let $d(A)=a$, $d(B)=b$, $d(C)=c$. Assume that $a\ge b\ge c>0$. Adversary answers in a way that the imbalance of the new component is not $a+b+c$. (This is the exact meaning of Condition 3, see Figure 1.) In other words, the larger classes of $A$, $B$ and $C$ should not get the same color. If any of $a, b, c$ is 0, then Condition 3 does not give any restriction for Adversary.

If we can answer in such a way, the imbalance of the new component is at most $a+b-c$.

\begin{tikzpicture}[line cap=round,line join=round,>=triangle 45,x=1.0cm,y=1.0cm, scale=0.4]
\clip(0.,-1.5) rectangle (36.,7.);
\draw [rotate around={0.:(4.,5.)},line width=0.8pt,color=qqqqff] (4.,5.) ellipse (2.061193392310586cm and 0.498515998243498cm);
\draw [rotate around={0.:(4.,2.)},line width=0.8pt,color=qqqqff] (4.,2.) ellipse (1.581138830084193cm and 0.5cm);
\draw [rotate around={0.:(9.,5.)},line width=0.8pt,color=qqqqff] (9.,5.) ellipse (2.0611933923105585cm and 0.4985159982434913cm);
\draw [rotate around={0.:(9.,2.)},line width=0.8pt,color=qqqqff] (9.,2.) ellipse (1.5811388300841978cm and 0.5cm);
\draw [rotate around={0.:(14.,5.)},line width=0.8pt,color=qqqqff] (14.,5.) ellipse (2.0611933923105723cm and 0.49851599824349463cm);
\draw [rotate around={0.:(14.,2.)},line width=0.8pt,color=qqqqff] (14.,2.) ellipse (1.5811388300842066cm and 0.5cm);
\draw [line width=0.8pt,color=ffqqqq] (1.9403454512869363,4.980740238960319)-- (2.420877586191564,1.9747563825677759);
\draw [line width=0.8pt,color=ffqqqq] (6.059654548713036,4.980740238960319)-- (5.579122413808435,1.9747563825677756);
\draw [line width=0.8pt,color=ffqqqq] (6.941126598796515,4.980097028628181)-- (7.421658733708327,1.9741131722358034);
\draw [line width=0.8pt,color=ffqqqq] (11.060435696235837,4.980097028628181)-- (10.579903561323754,1.974113172235804);
\draw [line width=0.8pt,color=ffqqqq] (11.941468210113245,4.980764778688362)-- (12.422000344999574,1.9747809222954156);
\draw [line width=0.8pt,color=ffqqqq] (16.060777307504992,4.980764778688362)-- (15.580245172618495,1.974780922295425);
\draw (3.7532933600852645,3.8248044629759406) node[anchor=north west] {$A$};
\draw (8.748378280799969,3.8248044629759406) node[anchor=north west] {$B$};
\draw (13.74346320151467,3.8248044629759406) node[anchor=north west] {$C$};
\draw [rotate around={0.:(22.,5.)},line width=0.8pt,color=qqqqff] (22.,5.) ellipse (2.0611933923106553cm and 0.4985159982435148cm);
\draw [rotate around={0.:(22.,2.)},line width=0.8pt,color=qqqqff] (22.,2.) ellipse (1.5811388300840807cm and 0.5cm);
\draw [rotate around={0.:(26.99866583296127,2.003797497826752)},line width=0.8pt,color=qqqqff] (26.99866583296127,2.003797497826752) ellipse (2.061193392311186cm and 0.49851599824364395cm);
\draw [rotate around={0.:(27.098665832961156,5.003797497826753)},line width=0.8pt,color=qqqqff] (27.098665832961156,5.003797497826753) ellipse (1.5811388300833265cm and 0.49999999999972766cm);
\draw [rotate around={0.:(32.,5.)},line width=0.8pt,color=qqqqff] (32.,5.) ellipse (2.0611933923102117cm and 0.49851599824340737cm);
\draw [rotate around={0.:(32.,2.)},line width=0.8pt,color=qqqqff] (32.,2.) ellipse (1.5811388300843325cm and 0.5cm);
\draw [shift={(26.99866583296122,-7.996202502173249)},line width=0.8pt,color=qqqqff]  plot[domain=1.1381927597298631:2.0032041023654337,variable=\t]({1.*14.314932457471159*cos(\t r)+0.*14.314932457471159*sin(\t r)},{0.*14.314932457471159*cos(\t r)+1.*14.314932457471159*sin(\t r)});
\draw [line width=0.8pt,color=ffqqqq] (20.99866583296123,5.003797497826752)-- (26.99866583296122,2.003797497826752);
\draw [line width=0.8pt,color=ffqqqq] (26.99866583296122,2.003797497826752)-- (33.,5.);
\draw [line width=0.8pt,color=ffqqqq] (19.940935287399306,4.989358153853864)-- (20.421467422379592,1.9833742974629727);
\draw [line width=0.8pt,color=ffqqqq] (24.060244384979754,4.989358153853864)-- (23.579712249999087,1.9833742974629691);
\draw [line width=0.8pt,color=ffqqqq] (29.93860228675193,4.980537377703517)-- (30.419134421380146,1.9745535213050018);
\draw [line width=0.8pt,color=ffqqqq] (34.057911383621196,4.980537377703517)-- (33.57737924899268,1.9745535213050034);
\draw [line width=0.8pt,color=ffqqqq] (25.520467195594666,5.034275458755888)-- (24.939717087779034,2.027056468797238);
\draw [line width=0.8pt,color=ffqqqq] (28.678539400120933,5.023796277234118)-- (29.058893915714574,2.0190526722258566);
\draw (21.735599074658193,3.8248044629759406) node[anchor=north west] {$A$};
\draw (26.757684454403787,3.8248044629759406) node[anchor=north west] {$B$};
\draw (31.752769375118486,3.8248044629759406) node[anchor=north west] {$C$};
\draw [->,line width=.5pt] (17.5,3.5) -- (18.5,3.5);
\begin{scriptsize}
\draw [fill=black] (3.,5.) circle (2.5pt);
\draw [fill=black] (9.,5.) circle (2.5pt);
\draw [fill=black] (15.,5.) circle (2.5pt);
\draw [fill=black] (20.99866583296123,5.003797497826752) circle (2.5pt);
\draw [fill=black] (26.99866583296122,2.003797497826752) circle (2.5pt);
\draw [fill=black] (33.,5.) circle (2.5pt);
\draw (18,0) node[anchor=north] {Figure 1};
\end{scriptsize}
\end{tikzpicture}

\begin{rem}
We note that Adversary can easily answer in a way that his answer does not violate Condition 2 and 3, however the problem lies in the fact that Adversary must not violate Condition 1 with his answer.

\end{rem}

\begin{prop}\label{components}

Assume that Adversary has not violated Condition 1, 2, and 3 till a certain point. Then the components of $G$ are the same as the components of $G_R$.

\end{prop}

\begin{proof}

The edges of $G_R$ are edges in $G$, so we only have to show that there are no blue edges between two components of $G_R$. Assume there is at least one. Consider the first answer, from which we get that kind of edge.

If this answer was that two of the balls (say $x$ and $y$) have the same color, the other (say $z$) has a different color, than $xy$ is the blue edge, but $x$ and $y$ are also in the same red component, because $z$ is a common neighbor of them in $G_R$.

If the answer was that the three balls have the same color, but the three balls did not came from the same component of $G_R$, than this answer violated Condition 2.

\end{proof}

\begin{prop}\label{number of edges}

Assume that the Adversary has not violated Condition 1, 2, and 3 till a certain point. Then the following two statements are true for every component $D$ of $G_R$. Recall that the two color classes of $D$ are $D_X$ and $D_Y$. Let us denote by $e_R(D)$ the number of red edges in $D$.

\vspace{2mm}

a) If $d(D)=0$, then $e_R(D) \ge |V(D)|=|D_X|+|D_Y|(=2|D_X|)$.

\vspace{1mm}

b) If $d(D)>0$, then $e_R(D) \ge|V(D)|+d(D)-2= 2\max(|D_X|,|D_Y|)-2$.

\end{prop}

\begin{proof}

a) $D$ is connected, so it has at least $|V(D)|-1=2|D_X|-1$ red edges. It is easy to see, that each component has even number of red edges, because every answer gives 0 or 2. Thus $e_R(D)$ is even and $e_R(D) \ge 2|D_X|-1$, so $e_R(D) \ge 2|D_X|=|V(D)|$.

\vspace{2mm}

b) We prove it by induction on $|V(D)|$. If $|V(D)|=1$, then $D$ is an isolated vertex, so $$0=e_R(D)\ge2\max(|D_X|, |D_Y|)-2=2\cdot1-2=0$$ is obviously true.

If a query comes from a component, then nothing is changed and we are done.

\vspace{2mm}

\textbf{Case 1:} Assume first that the query contains three balls from two components, call them $A$ and $B$. By induction we have $$e_R(A) \ge |V(A)|+d(A)-2,$$ and $$e_R(B) \ge |V(B)|+d(B)-2.$$ The answer of Adversary does not violate Condition 2 (by the assumption of the proposition), so $$e_R(A \cup B)=e_R(A)+e_R(B)+2.$$ Obviously we have $d(A \cup B) \le d(A)+d(B)$. By all of these, we get

$$e_R(A \cup B)=e_R(A)+e_R(B)+2 \ge$$
$$\ge |V(A)|+d(A)-2+|V(B)|+d(B)-2+2 \ge$$
$$\ge |V(A \cup B)|+d(A \cap B)-2,$$
and we are done in this case.

\vspace{2mm}

\textbf{Case 2:} Now assume that the query contains three balls from three components, call them 

$A$, $B$ and $C$, and let $d(A)=a$, $d(B)=b$, $d(C)=c$ (assume that $a \ge b \ge c$).

\vspace{2mm}

\hspace{2mm} \textbf{Case 2.1:} $c\ge 1$. 

\hspace{2mm}We will use that in the form $c-2 \ge -c$. By induction, we have $$e_R(A) \ge |V(A)|+d(A)-2,$$  $$e_R(B) \ge |V(B)|+d(B)-2,$$ 

\hspace{2mm}and $$e_R(B) \ge |V(C)|+d(C)-2.$$ 

\hspace{2mm}The answer does not violate Condition 2, so $$e_R(A \cup B \cup C)=e_R(A)+e_R(B)+e_R(C)+2,$$ 

\hspace{2mm}and Condition 3 provides that $$d(A \cup B \cup C) \le a+b-c.$$ 

\hspace{2mm}By all of these, we get
$$e_R(A \cup B \cup C)=e_R(A)+e_R(B)+e_R(C)+2 \ge$$
$$\ge |V(A)|+d(A)-2+|V(B)|+d(B)-2+|V(C)|+d(C)-2+2=$$
$$=|V(A \cup B \cup C)|+a+b+c-4 \ge$$
$$\ge |V(A \cup B \cup C)|+(a+b-c)-2 \ge$$
$$\ge |V(A \cup B \cup C)|+d(A \cup B \cup C)-2,$$

\hspace{2mm} and we are done in this case.

\vspace{2mm}

\hspace{2mm} \textbf{Case 2.2:} $c=0$. 

\hspace{2mm} In this case $e_R(B) \ge |V(C)|$ also holds by a). So we have

$$e_R(A \cup B \cup C)=e_R(A)+e_R(B)+e_R(C)+2 \ge$$
$$\ge |V(A)|+d(A)-2+|V(B)|+d(B)-2+|V(C)|+2=$$
$$=|V(A \cup B \cup C)|+a+b-2 =$$
$$= |V(A \cup B \cup C)|+(a+b-c)-2 \ge$$
$$\ge |V(A \cup B \cup C)|+d(A \cup B \cup C)-2,$$

\hspace{2mm}and we are done with this case and with the proof of Proposition \ref{number of edges}.

\end{proof}
\noindent
Proposition \ref{number of edges} immediately gives the following: 

\begin{cor}\label{dnot1} Assume that the Adversary has not violated Condition 1, 2, and 3 till a certain point.
If $d(D)\neq 1$, then $$e_R(D)\ge |V(D)|.$$

\end{cor}



\vspace{3mm}

\noindent
Now we go back to the proof of our main goal: to prove that there are at least $n-4$ red edges.

\vspace{2mm}

\textbf{Case 1:} Adversary never violates Condition 1, 2, and 3 during the algorithm. By Proposition \ref{endblue} and Proposition \ref{components}, it is easy to see, that there is only one component of $G_R$. By Corollary \ref{dnot1}, it has at least $|V(G_R)|=n>n-4$ red edges. And that was our goal.

\vspace{2mm}

\textbf{Case 2:} Adversary has to violate Condition 2 or 3 during the algorithm. Consider the first query, when he has to. From now on, we will use the notations $G$, $G_R$ and $G_B$ for the graphs before this query. We will show that $G_R$ already has at least $n-4$ edges.

\vspace{2mm}

\begin{tikzpicture}[line cap=round,line join=round,>=triangle 45,x=1.0cm,y=1.0cm, scale=0.4]
\clip(0.,-1.5) rectangle (33.,7.);
\draw [rotate around={0.:(4.,5.)},line width=2.pt,color=qqqqff] (4.,5.) ellipse (1.5811388300841887cm and 0.5cm);
\draw [rotate around={0.:(4.,2.)},line width=2.pt,color=qqqqff] (4.,2.) ellipse (2.06155281280883cm and 0.5cm);
\draw [line width=2.pt,color=ffqqqq] (2.420880672746554,5.025262918070062)-- (1.9399972976055266,2.0193859654021113);
\draw [line width=2.pt,color=ffqqqq] (6.060002702394472,2.019385965402111)-- (5.579119327253473,5.025262918070062);
\draw [rotate around={0.:(8.997194568313493,5.002055373959173)},line width=2.pt,color=qqqqff] (8.997194568313493,5.002055373959173) ellipse (2.061552812808869cm and 0.5cm);
\draw [rotate around={0.:(8.997194568313505,2.0020553739591733)},line width=2.pt,color=qqqqff] (8.997194568313505,2.0020553739591733) ellipse (1.581138830084283cm and 0.5cm);
\draw [line width=2.pt,color=ffqqqq] (6.937191865940886,4.982669408556438)-- (7.418075241061307,1.976792455888051);
\draw [line width=2.pt,color=ffqqqq] (10.576313895565736,1.976792455888048)-- (11.057197270686128,4.982669408556438);
\draw [rotate around={0.:(14.000997548694926,4.9950179672174135)},line width=2.pt,color=qqqqff] (14.000997548694926,4.9950179672174135) ellipse (2.0615528128089937cm and 0.5cm);
\draw [rotate around={0.:(14.,2.)},line width=2.pt,color=qqqqff] (14.,2.) ellipse (2.0603501754096705cm and 0.49501802523738647cm);
\draw [line width=2.pt,color=ffqqqq] (16.0625543732539,4.994929045580398)-- (16.06035014343667,1.9999127908524381);
\draw [line width=2.pt,color=ffqqqq] (11.939649824860368,1.9999919752723907)-- (11.93944172996487,4.9950097849157125);
\draw [rotate around={0.:(18.995301537902712,4.999109647439852)},line width=2.pt,color=qqqqff] (18.995301537902712,4.999109647439852) ellipse (2.061552812808885cm and 0.5cm);
\draw [rotate around={0.:(18.995301537902748,1.9991096474398504)},line width=2.pt,color=qqqqff] (18.995301537902748,1.9991096474398504) ellipse (1.5811388300841887cm and 0.5cm);
\draw [line width=2.pt,color=ffqqqq] (16.93529883554055,4.979723682036693)-- (17.41618221064828,1.9738467293680104);
\draw [line width=2.pt,color=ffqqqq] (20.574420865157627,1.973846729368006)-- (21.055304240265365,4.979723682036693);
\draw [rotate around={0.:(24.,5.)},line width=2.pt,color=qqqqff] (24.,5.) ellipse (1.5818773548663418cm and 0.5023305344487147cm);
\draw [rotate around={0.:(24.,2.)},line width=2.pt,color=qqqqff] (24.,2.) ellipse (2.060989435711475cm and 0.4976720346921731cm);
\draw [line width=2.pt,color=ffqqqq] (22.420148770165106,5.025416291453896)-- (21.94053820965522,2.0191580425818803);
\draw [line width=2.pt,color=ffqqqq] (26.05946179034487,2.019158042581881)-- (25.57985122983482,5.025416291453896);
\draw [rotate around={0.:(29.,5.)},line width=2.pt,color=qqqqff] (29.,5.) ellipse (1.580177089782026cm and 0.496950334610861cm);
\draw [rotate around={0.:(29.,2.)},line width=2.pt,color=qqqqff] (29.,2.) ellipse (2.0622951512536893cm and 0.5030519763257989cm);
\draw [line width=2.pt,color=ffqqqq] (27.421833604633726,5.0250616663049925)-- (26.939284753890238,2.019687236594405);
\draw [line width=2.pt,color=ffqqqq] (31.060715246109407,2.0196872365944043)-- (30.5781663953661,5.0250616663049925);
\draw (3.5062015000820566,3.8016238004467855) node[anchor=north west] {$A$};
\draw (8.494753364900113,3.8016238004467855) node[anchor=north west] {$B$};
\draw (13.509423302308841,3.8016238004467855) node[anchor=north west] {$C$};
\draw (8.494753364900113,5.812715389928365) node[anchor=north west] {$+b$};
\draw (18.4979751671269,5.812715389928365) node[anchor=north west] {$+1$};
\draw (23.512645104535626,2.80913704200133) node[anchor=north west] {$+1$};
\draw (18.4979751671269,1.320406904333148) node[anchor=north west] {$\mathcal{H}_{-1}$};
\draw (23.512645104535626,1.320406904333148) node[anchor=north west] {$\mathcal{H}_{+1}$};
\draw (25.99386200064932,6.80520214837382) node[anchor=north west] {$\mathcal{H}_{+}$};
\draw (3.5062015000820566,5.812715389928365) node[anchor=north west] {$\pm a$};
\draw (13.509423302308841,5.812715389928365) node[anchor=north west] {$\pm c$};
\draw [line width=2.pt] (16.5,6.)-- (16.5,1.);
\draw [line width=2.pt] (21.5,6.)-- (21.5,1.);
\draw [line width=2.pt] (26.5,5.5)-- (26.5,1.);
\draw (16.5,0) node[anchor=north] {Figure 2};
\end{tikzpicture}

 \hspace{4mm} \textbf{Case 2.1:} Assume that this query contains three balls from three components (note that before this query the components of $G$ and $G_R$ are the same by Proposition \ref{components}), call them $A$, $B$ and $C$, with $d(A)=a$, $d(B)=b$, and $d(C)=c$. We can assume that $a\ge b\ge c$. Certainly there is a balanced two-coloring $S \in \mathcal{S}$ consistent with this answer and the previous ones. Fix it, and denote the two color classes with $X$ and $Y$. For a component $D$, we will use the notation $D_X=V(D)\cap X$ and $D_Y=V(D)\cap Y$. 

Violating Condition 2 or 3 means that there is no balanced two-coloring $S' \in \mathcal{S}$ for which the balls in $A$ have the \textbf{same} colors as in $S$, and the balls in $B$ have the \textbf{other} colors as in $S$.

We can assume that $|B_X|\le |B_Y|$. Now we introduced another imbalance parameter: we give a sign for the imbalance. For a component $D$, the new definition of its imbalance is $$\od(D)=|D_X|-|D_Y|.$$ That means $\od(B)\le 0$.
Let $\mathcal{H}$ be the set of components of $G$ other than $A$, $B$ or $C$ and for $i \in \mathbb{Z}$ let
$$\mathcal{H}_{i}:=\{D\in \mathcal{H} \mid \od(D)= i\}.$$

In most cases we will change the color classes of some components i.e. exchange $D_X$ and $D_Y$, in order to get an other balanced two-coloring. (See Figure 2.)

\begin{prop}\label{Hb} We have

$$|\mathcal{H}_{1}|<b.$$

\end{prop}

\begin{proof}

If $|\mathcal{H}_1|\ge b$, then we could change the two color classes of $B$ and $b$ components from $\mathcal{H}_1$. But we saw that there are no such balanced two-coloring.

\end{proof}

\vspace{4mm}

\hspace{6mm} \textbf{Case 2.1.1:} $|\mathcal{H}_{-1}|\le a+c-1.$

\vspace{1mm}

\hspace{6mm}By Proposition \ref{number of edges} and Corollary \ref{dnot1}, there are at least $$n-|\cH_1|-|\cH_{-1}|+(a-2)+(b-2)+(c-2)$$ 

\hspace{6mm}red edges. Using $|\cH_{1}|+1\le b$ (Proposition \ref{Hb}) and $|\cH_{-1}|+1\le a+c$, we get that 

\hspace{6mm}there are at least
$$n+(b-|\cH_1|)+(a+c-|\cH_{-1}|)-6\ge n+1+1-6=n-4$$

\hspace{6mm}red edges.

\vspace{4mm}

\hspace{6mm} \textbf{Case 2.1.2:} $|\cH_{-1}|\ge a+c$

\vspace{2mm}

\hspace{6mm}Now let $$\cH_{+}:=\{D\in \cH \mid 0 < \od(D) \le |\cH_{-1}|+1 \}.$$

\begin{prop}\label{sumdless1} 

We have

$$\sum_{D\in \cH_{+}} \od(D)<b.$$

\end{prop}

\begin{proof}

Assume by contradiction that $\sum_{D\in \cH_{+}} \od(D)\ge b$. Then we first change the two color classes of $B$. The sum of the imbalances is $2b$. To compensate this, we can greedily change the color classes of some components $\cK \subseteq \cH_{+}$ such that 

$$b \le \sum_{D\in \cK} \od(D)\le b+|\cH_{-1}|.$$ 

After this, the total imbalance is $2k$, for some $k$ with $-2|\cH_{-1}|\le 2k\le 0$. And now we can again compensate this by changing the color classes of $k$ components from $\cH_{-1}$. We got a balanced two-coloring $S' \in \cS$ for which the balls in $A$ have the same colors as in $S$, and the balls in $B$ have the other colors as in $S$, which means Adversary should not have to violate Condition 2 or 3 with his answer, that is a contradiction.

\end{proof}


\vspace{4mm}

\hspace{8mm} \textbf{Case 2.1.2.1:} there is a component $D\in \cH$ with $$\od(D)>|\cH_{-1}|+1.$$

\vspace{2mm}

\hspace{8mm}By Proposition \ref{number of edges} and Corollary \ref{dnot1}, there are at least $$n-|\cH_1|-|\cH_{-1}|+(a-2)+(b-2)+(c-2)+(\od(D)-2)$$ 

\hspace{8mm}red edges. Using that $|\cH_{1}|+1\le b$ and $|\cH_{-1}|+2\le \od(D)$, we have at least
$$n+(b-|\cH_1|)+(\od(D)-|\cH_{-1}|)-8+a+c\ge n-5+a+c\ge n-4$$

\hspace{8mm}red edges. 


\vspace{4mm}

\hspace{8mm} \textbf{Case 2.1.2.2:} there is no component $D\in \cH$ with $$\od(D)>|\cH_{-1}|+1.$$

\vspace{2mm}

\hspace{8mm}$S$ is a balanced two-coloring, so we have
$$\sum_{D : \ \od(D)>0} \od(D)+\sum_{D: \ \od(D)<0} \od(D)=0.$$

\hspace{8mm}By the assumption we have $\{D \mid \od(D)>0\}\subseteq \cH_{+}\cup\{A, C\}$, so
$$\sum_{D: \ \od(D)>0} \od(D)\le \sum_{D\in \cH_{+}} \od(D)+a+c\le b-1+a+c.$$

\hspace{8mm}On the other hand we have $\cH_{-1}\cup\{B\}\subseteq\{D \mid \od(D)<0\}$, so
$$\sum_{D: \ \od(D)<0} \od(D)\le -|\cH_{-1}|-b\le -a-c-b.$$

\hspace{8mm}That gives us
$$0=\sum_{D: \ \od(D)>0} \od(D)+\sum_{D: \ \od(D)<0} \od(D) \le b-1+a+c-a-c-b=-1,$$

\hspace{8mm}and this is a contradiction.

\vspace{4mm}

\hspace{4mm} \textbf{Case 2.2:} Assume that the first violating query contains three balls from two components, call them $A$ and $B$. We can assume that $d(A)\ge d(B)$. 

\vspace{2mm}

The proof in this case is analogous to the case with three components. We fix the balanced two-coloring $S \in \mathcal{S}$ with color classes $X$ and $Y$.
Violating Condition 2 or 3 means that there aren't any balanced two-coloring $S' \in \mathcal{S}$ for which the balls in $A$ have the same colors as in $S$, and the balls in $B$ have the other colors as in $S$.
We define the signed imbalance parameter the same way as in Case 2.1, that means $\od(B)\le 0$.
Let now $\cH$ be the set of components of $G$ other than $A$ and $B$ and for $i \in \mathbb{Z}$ let 
$$\cH_{i}:=\{D\in \cH \mid \od(D)= i\}.$$

\begin{prop}\label{Hb2} We have

$$|\mathcal{H}_{1}|<b.$$

\end{prop}

\begin{proof} Same as the proof of Proposition \ref{Hb}.

\end{proof}


\vspace{4mm}

\hspace{6mm} \textbf{Case 2.2.1:} $|\cH_{-1}|\le a-1$

\vspace{2mm}

\hspace{6mm}By Proposition \ref{number of edges} and Corollary \ref{dnot1}, there are at least $$n-|\cH_1|-|\cH_{-1}|+(a-2)+(b-2)$$ 

\hspace{6mm}red edges. Using $|\cH_{1}|+1\le b$ (Proposition \ref{Hb2}) and $|\cH_{-1}|+1\le a$, we get that there 

\hspace{6mm}are at least

$$n+(b-|\cH_1|)+(a+c-|\cH_{-1}|)-4\ge n+1+1-4=n-2$$

\hspace{6mm}red edges.

\vspace{4mm}

\hspace{6mm} \textbf{Case 2.2.2:} $|\cH_{-1}|\ge a$.

\vspace{2mm}

\hspace{6mm}Now let $$\cH_{+}:=\{D\in \cH \mid 0<\od(D)\le |\cH_{-1}|+1 \}.$$ 

\begin{prop}\label{sumdless12} 

We have

$$\sum_{D\in \cH_{+}} \od(D)<b.$$

\end{prop}

\begin{proof}
The proof is the same as the proof of Proposition \ref{sumdless1}.
\end{proof}


\vspace{4mm}

\hspace{8mm} \textbf{Case 2.2.2.1:} there is a component $D\in \cH$ with $\od(D)>|\cH_{-1}|+1$.

\vspace{2mm}

\hspace{8mm}By Proposition \ref{number of edges} and Corollary \ref{dnot1}, there are at least $$n-|\cH_1|-|\cH_{-1}|+(a-2)+(b-2)+(\od(D)-2)$$ 

\hspace{8mm}red edges. Using $|\cH_{1}|+1\le b$ and $|\cH_{-1}|+2\le \od(D)$, we get that there are at least
$$n+(b-|\cH_1|)+(\od(D)-|\cH_{-1}|)-6+a\ge n-3+a\ge n-2$$

\hspace{8mm}red edges. 


\vspace{4mm}

\hspace{8mm} \textbf{Case 2.2.2.2:} there is no component $D\in \cH$ with $\od(D)>|\cH_{-1}|+1$.

\vspace{2mm}

\hspace{8mm}It is a balanced two-coloring, so
$$\sum_{D:\ \od(D)>0} \od(D)+\sum_{D:\ \od(D)<0} \od(D)=0.$$

\hspace{8mm}By the assumption $\{D \mid \od(D)>0\}\subseteq \cH_{+}\cup\{A\}$ we have
$$\sum_{D: \ \od(D)>0} \od(D)\le \sum_{D\in \cH_{+}} \od(D)+a\le b-1+a.$$

\hspace{8mm}On the other hand, $\cH_{-1}\cup\{B\}\subseteq\{D \mid \od(D)<0\}$, therefore

$$\sum_{D: \od(D)<0} \od(D)\le -|\cH_{-1}|-b \le -a-b.$$

\hspace{8mm}That gives us
$$0=\sum_{D:\ \od(D)>0} \od(D)+\sum_{D:\ \od(D)<0} \od(D) \le b-1+a-a-b=-1,$$

\hspace{8mm}and this is a contradiction.

\bigskip

We are done with the proof of the lower bound of Theorem \ref{m1a} \textbf{(i)}.

\qed

\subsection{Proof of the lower bound in Theorem \ref{m1a} \textbf{(ii)}}


Now we sketch the proof of Theorem \ref{m1a} \textbf{(ii)} that is similar to the one of \textbf{(i)}. These proofs are connected a similar way to the connection of the proof of the even and odd case in Theorem \ref{general}. The problem again lies in the fact that there is no exactly balanced two partition of odd many balls.

\vspace{2mm}

Now we divide Adversary's strategy into two phases. In Phase 1 he can violate Conditions 2-4 (see later) and wants to produce at least $n-O(\log_2 n)$ red edges in the auxiliary graph.

\vspace{2mm}

\textbf{Phase 1:} Before defining the strategy of Adversary we make some initial comments. 

\vspace{2mm}

During Phase 1 the auxiliary graph will be the same as in case of even $n$. We will call \textit{components} the components of $G_R$, but we will see that an analogue of Proposition \ref{components} is true, so the components of $G_R$ are the same as the components of $G$. The imbalance function $d$ will also be the same. 

We will call a component $D$ a \textit{deficient component}, if it has exactly $|V(D)|-1$ red edges. Note that if a component is not deficient, then it has more than $|V(D)|-1$ red edges.

We will call a ball of \textit{potential third color} (or p3c ball for short) if its blue degree is zero and it is contained in a deficient component. Note that an isolated vertex is a deficient component and it is p3c.

\vspace{2mm}

During Phase 1 Adversary will have four conditions on his strategy: 
\vspace{2mm}

\textbf{Condition 1:} there is a 3-coloring of the balls with color class sizes $\frac{n-1}{2}$, $\frac{n-1}{2}$, $1$, consistent with the answers. 

\vspace{2mm}

\textbf{Condition 2}, and \textbf{Condition 3} will be the same as in the case of even $n$.

\vspace{2mm}

\textbf{Condition 4:} there is an almost-balanced two-coloring of the balls (that is a three coloring, where the sizes of the color classes are $\frac{n+1}{2}$, $\frac{n-1}{2}$, 0), that is consistent with the answers and the larger contains a ball of potential third color.

Note that Condition 4 implies Condition 1. Indeed, if there is an almost-balanced two-coloring with color classes $X'$ and $Y'$ ($|X'|>|Y'|$), which is consistent with the previous answers, and there is a p3c vertex $z$ in the larger color class, then the three-coloring $X=X'\setminus\{z\}$, $Y=Y'$, $Z=\{z\}$ is consistent with Condition 1 and the previous answers. 


\vspace{2mm}

Now we define the answers of Adversary during Phase 1.

The answers will be the same as in the case of even $n$, except there is a little restriction in the case when the query contains three p3c balls from three different deficient components, which will we describe immediately.

But first note, that it is easy to see that a deficient component $D$ with more than one vertex, must have imbalance $d(D)=1$ and it had to be made from three different deficient components. (Any other way we get a component with more red edges. )

\vspace{2mm}

Now we define the answer of Adversary for a query containing three p3c balls from three different deficient components:

$\bullet$ If all the three p3c balls were on the larger size of their components, then Adversary answers that the one with the smallest red-degree has different color, and the other two have the same color.

$\bullet$ If all the three p3c balls were on the smaller size of their components, then Adversary answers that the one with the smallest red-degree has different color, and the other two have the same color.

$\bullet$ If exactly two of the three p3c balls were on the larger size of their components, then Adversary answers that the one (of the two) with the smaller red-degree has different color, and the other has the same color as the third ball.

$\bullet$ If exactly two of the three p3c balls were on the smaller size of their components, then Adversary answers that the one (of the two) with the smaller red-degree has different color, and the other has the same color as the third ball.

\vspace{2mm}

The following propositions are similar to Proposition \ref{components}, \ref{number of edges} and Corollary \ref{dnot1}, and the proofs are the same.

\begin{prop}\label{components 2}

Assume that Adversary has not violated Condition 1, 2, 3, and 4 till a certain point. Then the components of $G$ are the same as the components of $G_R$.

\end{prop}

\begin{proof} Similar to the proof of Proposition \ref{components}.

\end{proof}

\begin{prop}\label{number of edges 2}

Assume that Adversary has not violated Condition 1, 2, 3, and 4 till a certain point. Then the following two statements are true for every component $D$ of $G_R$. Recall that the two color classes of $D$ are $D_X$ and $D_Y$. Let us denote by $e_R(D)$ the number of red edges in $D$.

\vspace{2mm}

a) If $d(D)=0$, then $e_R(D) \ge |V(D)|=|D_X|+|D_Y|(=2|D_X|)$.

\vspace{1mm}

b) If $d(D)>0$, then $e_R(D) \ge|V(D)|+d(D)-2= 2\max(|D_X|,|D_Y|)-2$.

\end{prop}

\begin{proof} Similar to the proof of Proposition \ref{number of edges}.

\end{proof}
\noindent
Proposition \ref{number of edges 2} immediately gives the following: 

\begin{cor}\label{dnot2} Assume that Adversary has not violated Condition 1, 2, 3, and 4 till a certain point.
If $d(D)\neq 1$, then $$e_R(D)\ge |V(D)|.$$

\end{cor}

\begin{lemma}\label{logdegree}

Assume that Adversary has not violated Condition 1, 2, 3, and 4 till a certain point. Then the following two statements are true for every deficient component $D$. 

\vspace{2mm}

a) There exists a p3c vertex in $D$.

\vspace{1mm}

b) There exists a p3c vertex in $D$ with red-degree at most $2\log_2(|V(D)|)$.

\end{lemma}

\begin{proof}

a) There are $|V(D)|$ vertices in $D$, so there are $|V(D)|-1$ red edges. Thus there were $\frac{|V(D)|-1}{2}$ questions that made this component. Hence Condition 2 provides us that there are $\frac{|V(D)|-1}{2}$ blue edges in $D$. That means there must be at least one vertex that has blue-degree 0, so it is a p3c vertex.

b) We will prove that by induction. For an isolated vertex it is true. Assume that there were a query containing balls from components $D_1$, $D_2$, and $D_3$. If any of the balls was not p3c, then in its component there will be a p3c ball which will be good. If all the three balls are p3c, then Adversary has the freedom to choose the ball with the different color (the one that will get two red edges) from at least two balls. Assume that this two ball is from $D_1$ and $D_2$, so by induction, their red-degree is at most $2\log_2(|V(D_1)|)$ and $2\log_2(|V(D_2)|)$. Adversary can choose the one with the smaller degree, so it will be a p3c ball, and its red degree is at most
$$\min(2\log_2(|V(D_1)|),2\log_2(|V(D_2)|))+2 = 2(\log_2(\min(|V(D_1)|,|V(D_2)|))+1) =$$
$$= 2\log_2(2\min(|V(D_1)|,|V(D_2)|)) \le 2\log_2(|V(D_1)|+|V(D_2)|) \le$$
$$\le 2\log_2(|V(D_1)|+|V(D_2)|+|V(D_3)|) = 2\log_2(|V(D)|).$$

\end{proof}

In a coloring provided by Condition 1, there is no plurality ball. In a coloring provided by Condition 4, there are plurality balls. So if the Questioner can solve the Plurality problem, then Adversary must violate some Condition at some point, however it cannot be Condition 1. Now we will consider the first answer, where the Adversary had to violate Condition 2, 3, or 4.

\begin{lemma} \label{viol23}

Assume that Adversary has not violated Condition 1, 2, 3, and 4 till a certain point, but he has to violate Condition 2 or 3 with his next answer for a query. Then we already have $n-10$ red edges.

\end{lemma}

\begin{proof}[Sketch of the proof]
The proof is very similar to the proofs of \textbf{Case 2.1} and \textbf{Case 2.2} in the proof of the even case. 

We just highlight some differences that create only some modification of the constants: 

$\bullet$  in the analogue of Proposition \ref{Hb} and Proposition \ref{Hb2} we need $|\cH_1|<b+1$, because we need one p3c vertex in the larger class after the color-changing. 

$\bullet$ at the end, instead of $$0=\sum_{D : \ \od(D)>0} \od(D)+\sum_{D: \ \od(D)<0} \od(D),$$ 

we need to use $$-1\le\sum_{D : \ \od(D)>0} \od(D)+\sum_{D: \ \od(D)<0} \od(D).$$

\end{proof}

\begin{lemma} \label{viol4}

Assume that Adversary has not violated Condition 1, 2, 3, and 4 till a certain point, but he has to violate Condition 4 with his next answer for a query. Then there is at least $n-5$ red edges in $G$.

\end{lemma}

\begin{proof}
Note that there are two ways of violating Condition 4:

\vspace{2mm}
1) there is no almost-imbalanced two-coloring of the balls consistent with the previous answers and an answer according to the Adversary's strategy.

\vspace{2mm}

2) There are such almost-imbalanced two-colorings, but in any of these colorings there are no p3c ball in the larger color class.

\vspace{2mm}

Suppose by contradiction that 1) would happen. Let $\cT$ be the set of those two-colorings of the balls that are consistent with the previous answers and with any possible answer for that query. 

By the assumption, every two-coloring in $T \in \cT$ has imbalance at least $2$. Fix a two-coloring $T \in \cT$ with minimal imbalance $e$. Denote the two color-classes by $X$ and $Y$, where $|X|=\frac{n+e}{2}$ and $|Y|=\frac{n-e}{2}$. 

\vspace{2mm}

$\bullet$ Denote by $\cH_1$ the set of components with imbalance $1$ and with larger class in $X$.  
If there exists $D \in \cH_1$, then we can change the color-classes of $D$, and we get a two-coloring $\cT$ with imbalance $e-2$, which is a contradiction. So $\cH_1=\emptyset$.

\vspace{2mm}

$\bullet$ We know that $|X|>|Y|$, so there must be a component with larger class in $X$, denote it by $E$ (and its imbalance by $d(E)$). If we change the color classes of $E$, the size of the classes of the new two-coloring are $\frac{n+e}{2}-d(E)$ and $\frac{n-e}{2}+d(E)$. The previous cannot be the larger, because it would give us a two-coloring from $\cT$ with imbalance $e-2d(E)<e$. So, $\frac{n-e}{2}+d(E)>\frac{n+e}{2}-d(E)$ and the imbalance of the new two-coloring is $2d(E)-e\ge e$.

But we can say more. Denote by $\cH_{-1}$ the set of components with imbalance $1$ and with larger class in $Y$. If we can change the color-classes of some components from $\cH_{-1}$, we still cannot go under the imbalance of $e$. So if we change every component in $\cH_{-1}$, the size of the color-classes are $\frac{n+e}{2}-d(E)+|\cH_{-1}|$ and $\frac{n-e}{2}+d(E)-|\cH_{-1}|$, where the latter is the larger. Thus the imbalance of that two-coloring is
$$\left(\frac{n-e}{2}+d(E)-|\cH_{-1}|\right)-\left(\frac{n+e}{2}-d(E)+|\cH_{-1}|\right)=2d(E)-2|\cH_{-1}|-e.$$
This cannot be larger than $e$, so we have $2d(E)-2|\cH_{-1}|-e\ge e$, which gives us
$$d(E) \ge |H_{-1}|+e \ge |H_{-1}|+2.$$

Now we count the red edges. In every component $D\not\in \cH_{1}\cup\cH_{-1}\cup\{E\}$, there are at least $|V(D)|$ red edges. In every component $D\in \cH_{1}\cup\cH_{-1}$, there are at least $|V(D)|-1$ red edges. In $E$ there are at least $|V(E)|+(d(E)-2)$ red edges. That gives us altogether at least
$$n-|\cH_{1}|-|\cH_{-1}|+(d(E)-2) \ge n-0-|\cH_{-1}|+|\cH_{-1}|+2-2=n$$
red edges.

It is possible that we counted some extra red edges because of our potential (last) answer. It is easy to see, that this error is at most 4, so before the answer we already had $n-4$ red edges.

\vspace{2mm}

Suppose we are in case 2). As in the previous case, we consider a set $\cT$ of two-colorings with imbalance 1 that are consistent with the previous answers and a potential answer for that query. By the assumption, there is no two-coloring in $T \in \cT$ with p3c balls in the larger color class. $\cT$ is not empty, because that would give us case 1). Fix a two-coloring from $\cT$ with color classes $X$ and $Y$, where $|X|=|Y|+1$. Again, denote by $\cH_1$ the set of components with imbalance 1 and larger class in $X$. Denote $\cH_{-1}$ the set of components with imbalance 1 and larger class in $Y$.

\vspace{2mm}

$\bullet$ If there exists $A\in\cH_1$ and $B\in\cH_{-1}$, then we can change the color classes of $A$ and $B$. This would give us a two-coloring $T \in \cT$, but there would be two p3c balls in the larger color class, which is a contradiction. So either $\cH_1=\emptyset$ or $\cH_{-1}=\emptyset$.

\vspace{2mm}

$\bullet_1$ If $\cH_{-1}=\emptyset$ and $A$ and $B$ are two distinct components from $\cH_1$, then we could change the color classes of $A$, which would give us a two-coloring from $T$. But now the other class is the bigger, where we can find a p3c ball in $B$. This is a contradiction, so if $\cH_{-1}=\emptyset$, then $|\cH_{1}|\le 1$. That gives us we have at least $$n-|\cH_{-1}|-|\cH_{1}|\ge n-1$$ red edges. It is possible that we counted some (at most 4) extra red edges with the potential (last) answer, so there are at least $n-5$ red edges in $G$ before the answer.

\vspace{2mm}

$\bullet_2$ Now assume that $\cH_1=\emptyset$. $X$ is the larger color class, so there must be a component $D$ with its larger class in $X$. Obviously $D\not\in \cH_{-1}$. (Let us denote the imbalance of $D$ by $d(D)$.)

$\bullet_{2.1}$Assume that $d(D)\le |\cH_{-1}|$. Then we can change the color classes of $D$ and $d(D)$ components from $\cH_{-1}$ to get a two-coloring from $T$. But there would be $d(D)$ p3c balls in the larger color class $X$, which is a contradiction. 

$\bullet_{2.2}$ We have $d(D)\ge |\cH_{-1}|+1$. Then the number of red edges is at least
$$n-|\cH_{1}|-|\cH_{-1}|+(d(D)-2) \ge n-1.$$

Again, we possibly counted at most 4 extra edges because of the potential last answer, so before it, there were at most $n-5$ red edges.

We are done with the proof of Lemma \ref{viol4}.

\end{proof}

As we have already mentioned, in a coloring provided by Condition 1, there is no plurality ball. In a coloring provided by Condition 4, there are plurality balls. So if the Questioner can solve the Plurality problem, then Adversary must violate Condition 4 at some point. 

Now consider the first query $Q$, for which he had to answer in a way that he violated Condition 2, 3, or 4. This is the end of Phase 1. Before $Q$, there were at least one almost-imbalanced two-coloring with at least one p3c ball in the larger color class. Fix such a two-coloring ($[n]=X'\cup Y'$, $|X'|=\frac{n+1}{2}$, $|Y'|=\frac{n-1}{2}$), and choose a p3c ball $z\in X'$ with red-degree at most $2\log_2(n)$. Such a ball exists according to Lemma \ref{logdegree} b) and because of $|V(D)|\le n$.

As we have noted, the three-coloring $X=X'\setminus\{z\}$, $Y=Y'$, $Z=\{z\}$ is consistent with the previous answers and Condition 1. Now we fix this three-coloring, and Adversary will answer according to it starting with the answer for $Q$.

We also make some changes in the auxiliary graph. Note that the red edges from $z$ are in pairs because they come from the queries involving $z$. Let us denote them by $e_{1,1},e_{1,2},e_{2,1},e_{2,2},\dots, e_{l,1},e_{l,2}$, where $2l$ is the red-degree of $z$. Note that we have $l\le \log_2(n)$.

Now we define a new color of the edges. Green edge simply means that it comes from $z$. Basically, they are special red edges.

We change the colors of $e_{1,1},e_{2,1},\dots, e_{l,1}$ from red to green, and we delete the edges $e_{1,2},e_{2,2},\dots, e_{l,2}$ from the graph.

It is the end of Phase 1. By Lemma \ref{viol23} and \ref{viol4}, there were at least $n-10$ red edges before $Q$, so now we have at least $n-10-2l\ge n-10-2\log_2(n)$ red edges.

\vspace{2mm}

\textbf{Phase 2:}
After $Q$, Adversary must not violate Condition 1, but he can violate again any of Condition 2, 3, or 4. We can assume that his answers correspond to the previously fixed coloring $X$, $Y$, $Z$.

\vspace{2mm}

We will have some new rules in the auxiliary graph in Phase 2. If the Questioner asks the query $\{u, v, z\}$, where $u$ and $v$ have the same color, then Adversary adds the edge $uv$ and colors it blue, and adds one of $uz$ and $vz$, colored green. (That is also how he changed the auxiliary graph at the end of Phase 1.) If Questioner asks the query $\{u, v, z\}$, where $u$ and $v$ have different colors, he adds $uz$ and $vz$ and colors both edges green, but does not add the edge $uv$. If Questioner asks a query disjoint from $z$, the rules do not change. That allows him to give the green edges weight $\frac{1}{2}$, so every answer has a total weight of 1.

We proved that at the end of Phase 1, there were at least $n-2\log_2 n-10$ red edges, which did not decrease in Phase 2. Note that we can apply Lemma \ref{oddgreenblue} also in this setting, thus at the end of Phase 2 there are altogether at least $n-5$ blue and green edges. Give weight $\frac{1}{4}$ to the red edges and $\frac{1}{2}$ to the blue and green edges, so every answer has weight 1. That means the number of questions is at least $\frac{1}{4}(n-2\log_2(n)-10)+\frac{1}{2}(n-5)=\frac34n-\frac12\log_2(n)-5$ which proves Theorem \ref{m1a} \textbf{(ii)}.

\qed






\section{Remarks, Questions}

$\bullet$ Our Theorem \ref{m1a} \textbf{(i)} gives an exact result for $n=4k+2$, and an almost exact result (it can be one of two values) for $n=4k$, while Theorem \ref{m1a} \textbf{(ii)} gives an asymptotic result for $n$ odd. So first of all it would be very interesting to prove an exact result for $n=4k, 4k+1, 4k+3$. Theorem \ref{general} provides bounds for $k \ge 4$. It would be nice to prove some asymptotics or even sharper results for $A_p(n,k)$ for $k \ge 4$.
\vspace{2mm}

\noindent
$\bullet$ Our model is just one of the possible models with three colors. One can also imagine other answers, or the model can be non-adaptive. 

\vspace{2mm}

\noindent
$\bullet$ In this article we investigated a model with three colors, however all Plurality problems can be asked with more colors. We note that the more colors one has, the more models can be imagined, as there can be more possible answers.

\vspace{2mm}

\noindent
$\bullet$ If we have two colors, then - as we mentioned in the introduction - we get the so called Majority problem. It would be interesting to apply (some variant of) the techniques we introduced here for those problems.



\section*{Acknowledgement}

We would like to thank all participants of the Combinatorial Search Seminar at the Alfr\'ed R\'enyi Institute of Mathematics for fruitful discussions.

\vspace{3mm}

\noindent
Research of Gerbner was supported by the J\'anos Bolyai Research Fellowship of the Hungarian Academy of Sciences and by the National Research, Development and Innovation Office -- NKFIH, grant K 116769.

\noindent
Research of Vizer was supported by the National Research, Development and Innovation Office -- NKFIH, grant SNN 116095.

\end{document}